%% file: main.tex
\documentclass{article}
\usepackage{graphicx}
\usepackage{amssymb}
\usepackage{amsmath}
\usepackage{wasysym}
\usepackage[normalem]{ulem}
\usepackage{fullpage}
\usepackage{bbm}
\usepackage{xspace}
\usepackage{epic,eepic}
\usepackage{tikz}
\usepackage{subcaption}
\usepackage{color, colortbl}
\usepackage{gastex}
\usetikzlibrary {positioning}
\usepackage[]{algorithm2e}
\usepackage{amsthm}
\usetikzlibrary{arrows,automata,positioning}
\usepackage{amsmath,amssymb,amsfonts}
\usepackage{mwe}

\usepackage[T1]{fontenc}
\usepackage[utf8]{inputenc}
\usepackage{lmodern}
\usepackage[english]{babel}
\usepackage[autostyle]{csquotes}

\usepackage{fullpage}
\usepackage[ocgcolorlinks,pdfusetitle]{hyperref}
\usepackage{doi}

\newtheorem{theorem}{Theorem}
\newtheorem*{theorem*}{Theorem}
\newtheorem{proposition}{Proposition}
\newtheorem*{proposition*}{Proposition}
\newtheorem{problem}{Problem}
\newtheorem*{problem*}{Problem}
\newtheorem{conjecture}{Conjecture}
\newtheorem*{conjecture*}{Conjecture}
\newtheorem{heuristic}{Heuristic}
\newtheorem{claim}{Claim}
\newtheorem*{claim*}{Claim}

\newtheorem*{claimproof*}{Proof}

\newtheorem*{corollary*}{Corollary}

\newtheorem*{corollaryproof*}{Proof}
\newtheorem{lemma}{Lemma}
\newtheorem*{lemma*}{Lemma}

\newtheorem*{lemmaproof*}{Proof}

\definecolor {processblue}{cmyk}{0.96,0,0,0}

\newcommand{\note}[1]{}

\newcommand{\stam}[1]{}

\newcommand{\set}[1]{\left\{#1\right\}}

\usepackage{tabularx}

\newcommand\clearrow{\global\let\rowmac\relax}
\clearrow

\title{Efficient Generation of One-Factorizations through Hill Climbing}
\author{Maya Dotan\thanks{Department of Computer Science, Hebrew University, Jerusalem 91904, Israel. email: maya.dotan1@mail.huji.ac.il~.}\and {Nati Linial\thanks{Department of Computer Science, Hebrew University, Jerusalem 91904, Israel. email: nati@cs.huji.ac.il~. Supported by ERC grant 339096 "High dimensional combinatoric".}}}

\date{}
\begin{document}
\maketitle{}

\begin{abstract}
It is well known that for every even integer $n$, the complete graph $K_{n}$ has a one-factorization, namely a proper edge coloring with $n-1$ colors. Unfortunately, not much is known about the possible structure of large one-factorizations. Also, at present we have only woefully few explicit constructions of one-factorizations. Specifically, we know essentially nothing about the {\em typical} properties of one-factorizations for large $n$. 

Suppose that $\cal C_{\rm n}$ is a graph whose vertex set includes the set of all order-$n$ one-factorizations and that $\Psi: V(\cal C_{\rm n})\to \mathbb R$ takes its minimum precisely at the one-factorizations. Given $\cal C_{\rm n}$ and $\Psi$, we can generate one-factorizations via {\em hill climbing}. Namely, by taking a walk on $\cal C_{\rm n}$ that tends to go from a vertex to a neighbor of smaller $\Psi$. For over 30 years, hill-climbing has been essentially the only method for generating many large one-factorizations. However, the validity of such methods was supported so far only by numerical evidence. Here, we present for the first time hill-climbing algorithms that provably generate an order-$n$ one-factorization in $\text{polynomial}(n)$ steps regardless of the starting state, while all vertex degrees in the underlying graph are appropriately bounded.

We also raise many questions and conjectures regarding hill-climbing methods and concerning the possible and typical structure of one-factorizations.
\end{abstract}

\section{Introduction}
It is a very old result (e.g. \cite{Hall67}) that for every even integer $n$, the edges of the complete graph $K_{n}$ can be properly colored with $n-1$ colors. Such a coloring is called a {\em one-factorization}, and there is a considerable body of research dedicated to their study, e.g., \cite{mendelsohn1985one,wallis2013one}. A one-factorization is often viewed as a schedule for the games in a league of $n$ teams. Clearly, each of the $n-1$ color classes in a one-factorization is a perfect matching of $n/2$ edges. Accordingly, one speaks of $n-1$ rounds of games in each of which the $n$ teams are paired up to play. If the edge $ij$ is colored $k$, this means that teams $i$ and $j$ meet at round $k$. We prefer to view a one-factorization as a symmetric Latin square or, equivalently, a symmetric two-dimensional permutation (see \cite{linial2014upper}). This places the investigation of one-factorizations in the context of high-dimensional combinatorics.

We use the shorthand OF for one-factorization and we write OF$_n$ to indicate the order of the complete graph that is being factored. 

As we describe in Section \ref{section:expect}, many questions about the extremal and typical properties of OFs suggest themselves, but we presently lack the necessary tools to attack them. In particular, we need methods to generate OFs randomly, preferably uniformly, or at least with good control over the distribution of the generated OFs. Unfortunately, this goal seems presently out of reach. In fact, even the problem of systematically generating OFs (regardless of the distribution) is still poorly understood. Such methods and their analysis are the main technical contributions of the present article.

All the existing approaches to the systematic generation of OFs depend on the method of hill climbing as described, e.g., in \cite{dinitz1987hill}. They come with no proofs or guarantees for their run time, and our purpose is to remedy the situation and give hill-climbing generators with a guaranteed bound on their run time.

\subsection{The different flavors of hill-climbing}

Suppose that $W$ is a set of objects that we want to generate. A hill-climbing method
to do this works as follows: We define a graph $G=(V,E)$ with $W\subseteq V$ and a {\em potential function} $\Psi:V\to\mathbb R$ with the property that the restriction of $\Psi$ to $W$ is a constant function, and $\Psi(x)>\Psi(y)$ for all $x\in V\setminus W$ and $y\in W$. Hill climbing is carried out by taking a walk on $G$. Roughly, we consider three types of such walks:
\begin{itemize}
\item 
{\sl Strict:} We always move from a vertex $v \in V$ to a neighbor $u$ with $\Psi(v)>\Psi(u)$.
\item
{\sl Mild:}
Same, with $\Psi(v)\ge\Psi(u)$.
\item
{\sl Weak:} Most steps are from $v$ to neighbor $u$ with $\Psi(v)\ge\Psi(u)$, but occasional moves that increase $\Psi$ are allowed. 
\end{itemize}

Accordingly, we also speak of the {\em strict random walk} on $G$. It moves at each step from the current vertex $v$ to a neighbor that is chosen uniformly at random from among those that satisfy $\Psi(v)>\Psi(u)$. Likewise, in the {\em mild random walk} on $G$ the next vertex is chosen uniformly from among the neighbors of $v$ with $\Psi(v)\ge\Psi(u)$.

We are now able to state the main results of the present paper. We mostly work with the graph $\mathcal{G}_n$ that has $(n-1)^{\binom n2}$ vertices, that represent all (not necessarily proper) colorings of $E(K_n)$ with colors $\{1,\ldots,n-1\}$. Two vertices in $\mathcal{G}_n$ are adjacent if the corresponding colorings differ on exactly one edge. We define the potential function $\Psi(C)$ to be {\em the number of pairs of incident edges that are equally colored} in $C$. It is easy to see that $\Psi(C)=0$ iff $C$ is a OF, in which case $C$ is a {\em sink} in $\mathcal{G}_n$, i.e., each of its neighbors $C'$ satisfies $\Psi(C')>\Psi(C)$. Here are our main results:

\begin{theorem}\label{thm:weak}
There is a weak hill climbing algorithm on $\mathcal{G}_n$ that arrives from every starting point to a one-factorization in $O(n^4)$ steps. If the walk visits $u\in V(\mathcal{G}_n)$ and later visits $v$, then $\Psi(u)+B\ge\psi(v)$ for some absolute constant $B$. (e.g., $B=4$ suffices).
\end{theorem}

In the following theorem we do the hill climb on the following, more complicated graph $\cal{D}\rm_n$. It has the same vertex set as $\cal{G}\rm_n$, namely all the colorings of $E(K_n)$ with colors $\{1,\ldots,n-1\}$. Here two vertices $C$ and $C'$ are adjacent if there are are two vertices $u, v \in V(K_n)$ such that every edge in $E(K_n)$ on which the colorings $C$ and $C'$ differ is incident with $u$ or with $v$.

\begin{theorem}\label{thm:strict}
There is a strict hill-climbing algorithm on $\cal{D}\rm_n$ that arrives from every starting point to a one-factorization in $O(n^3)$ steps.
\end{theorem}

The following natural conjecture suggests itself. It is supported by ample numerical evidence, see Figure \ref{fig:metropolis}.

\begin{conjecture}\label{cnj:mild}
The mild random walk on $\mathcal{G}_n$ started from a uniformly random starting point asymptotically almost surely reaches a one-factorization in $O(n^4)$ steps.
\end{conjecture}

It may very well be that the same holds even with an arbitrary staring point.

\begin{figure}[h!]
  \centering
  \includegraphics[width=0.6\textwidth]{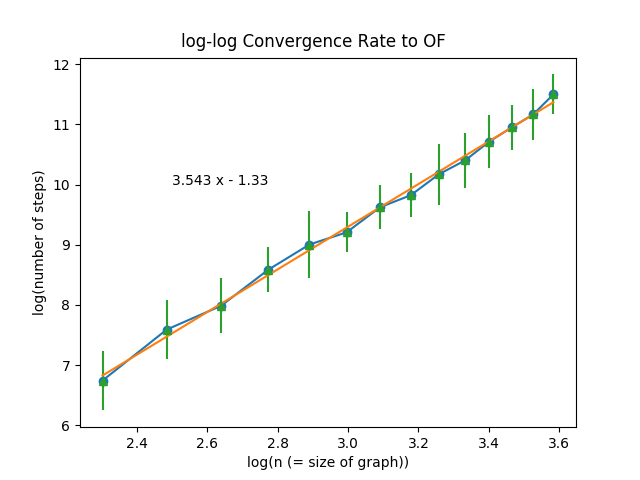}
   \caption{Convergence rate of the mild random walk on $\mathcal{G}_n$}
  \label{fig:metropolis}
\end{figure}

\section{Polynomial-time hill-climbing generative algorithms for OF$_n$'s}\label{sec:pf}

The purpose of this section is to prove Theorems \ref{thm:weak} and \ref{thm:strict}. We start with a part of the analysis that is common to both proofs. Namely, we consider what happens when we take the strict random walk on $\mathcal{G}_n$ until we reach a sink. We refer to such a state that is a local minimum of $\Psi$ as a {\em locally optimal} coloring.\\
Note that a strict walk on $\mathcal{G}_n$ is also a strict walk on $\cal{D}\rm_n$, since $\mathcal{G}_n$ is a subgraph of $\cal{D}\rm_n$, and in both cases the same potential function is applied.

\subsection{Dealing with locally optimal states}
If we insist on a strict walk, as in Theorem \ref{thm:strict}, and if the walk presently resides in a locally optimal coloring, we clearly need a different kind of modification step. To this end we introduce a {\em two-vertex recoloring}, in which we recolor some edges that are incident with two vertices. Our main technical statement is:

\begin{theorem}\label{thm:main}
If $C$ is a locally optimal coloring of $E(K_n)$ that is not a one-factorization, then there are two vertices $x, y \in [n]$ and a recoloring of the edges incident with $x, y$ such that the resulting coloring $C'$ satisfies $\Psi(C)>\Psi(C')$. Furthermore, it is possible to find the vertices $x,y$ and the said recoloring in time poly($n$).
\end{theorem}

The first step in proving the theorem is to show, in the next lemma, that locally optimal colorings have a quite restricted structure. In such a coloring every monochromatic connected component is either a single edge or a path with two edges. We refer to such a coloring as an {\em IV coloring}. We call a two-edge path a {\em Vee}. A Vee has a {\em center} and two {\em ends}. Unless otherwise stated, when we speak of a Vee, we implicitly assume that it is {\em monochromatic}. A Vee whose two edges are colored $\alpha$ is denoted by Vee$_{\alpha}$. Clearly, some colors are missing at the Vee's center vertex and if $\beta$ is such a missing color we may refer to it Vee$_{\alpha}^{\beta}$. (Note that $\beta$ need not be uniquely defined. There may be several colors missing at a Vee's center). We also use the notation Vee$^{\beta}$ to indicate only that color $\beta$ is missing at the Vee's center.

If $a_{u, \mu}$ is the number of $\mu$-colored edges that are incident with $u$ in some edge-coloring $C$ of $K_n$, then $\Psi(C)=\sum_{u,\mu}\binom{a_{u,\mu}}{2}=\frac{1}{2}\sum_{u,\mu}a_{u,\mu}^2-\binom{n}{2}=\frac{\Phi(C)}{2}-\binom{n}{2}$, where $\Phi(C)=\sum_{u,\mu}a_{u,\mu}^2$. Consequently it is immaterial whether we work with either $\Phi$ or $\Psi$ and we freely switch between the two. We also use the notation $\Phi(u)=\sum_{\mu}\binom{a_{u,\mu}}{2}$.

Here is the description of a locally-optimal coloring.

\begin{lemma}\label{shafan}
In a locally optimal coloring every color class is the vertex-disjoint union of edges and Vees.
\end{lemma}

\begin{proof}

We consider how $\Phi$ changes as we recolor the edge $uv$ from $i$ to $j$. Let us denote $a_{u, \mu}$ by $x_{\mu}$ and $a_{v, \mu}$ by $y_{\mu}$. The only  $a_{\nu,\mu}$ that are affected by this change are $x_i, x_j, y_i$ and $y_j$. The following inequality expresses the condition that $\Phi$ does not decrease as a result of this change \[{x_i}^2+{x_j}^2+{y_i}^2+{y_j}^2 \le ({x_i-1})^2+({x_j+1})^2+({y_i-1})^2+({y_j+1})^2.\] That is $x_i + y_i \le x_j + y_j +2$. We sum this inequality over all $j\neq i$ and use the fact that $\sum_{\mu} x_{\mu} = \sum_{\mu} y_{\mu} = n-1$ to show that $(n-1)\cdot (x_i + y_i) \le 4n-6$. Since these variables are positive integers we conclude that $(x_i + y_i) \le 3$. This implies that an $i$-colored edge can be incident to at most one other $i$-colored edge. The conclusion follows.
\end{proof}

The number of Vees in a locally optimal coloring $C$ is $\Psi(C)$. Therefore reducing $\Phi$ for a locally optimal coloring, is synonymous with a reduction in its number of Vees. Let us consider an IV coloring $C$ that is not an OF$_n$. Since $C$ is not an OF, it contains some monochromatic Vee, say Vee$_{\alpha}$. But $n$ is even, so if $C$ has exactly one Vee$_{\alpha}$, then necessarily there must be a vertex that is incident to no $\alpha$-colored edges. Such a vertex must, in turn, be the center of some Vee, say a Vee$_{\gamma}$ for some $\gamma\neq\alpha$. To sum up, if $C$ is IV but not an OF$_n$, and if $C$ has a monochromatic Vee$_{\alpha}$, then $C$ must
\begin{enumerate}

\item\label{no_alpha}
Contain a Vee$_{\gamma}^{\alpha}$ for some $\gamma\neq\alpha$. (See Figure: \ref{fig:First_Fix_Case}), or
\item
Contain an additional Vee$_{\alpha}$ (See Figure: \ref{fig:Second_Fix_Case})
\end{enumerate}

\begin{figure}[!ht]
						\begin{minipage}{.5\textwidth}
							\centering
							\input{First_Fix_Case.tex}
							\centering
							\caption{$\text{There is a Vee$_{\gamma}^{\alpha}$ in $C$}$}
							\label{fig:First_Fix_Case}
						\end{minipage}
						\begin{minipage}{.5\textwidth}
							\centering
							\input{Second_Fix_Case.tex}
							\centering
							\caption{$\text{$C$ contains an additional Vee$_{\alpha}$}$}
							\label{fig:Second_Fix_Case}
						\end{minipage}			
\end{figure}

Fix an IV coloring $C$ that is not an OF$_n$. As it turns out, one of the following happens now. Either we can find a two-vertex recoloring step that reduces $\Phi$, or we can recolor an edge in a way that does not change $\Phi$ and then do a beneficial two-vertex recoloring step. In our search for a two-vertex recoloring we actually restrict ourselves to a {\em flip}, a special kind of a two-vertex recoloring step as we now describe. Let $u$ and $v$ be the two vertices whose edges we intend to recolor. In a $(u,v)$-flip, for every vertex $w\neq u,v$ either the edges $wu$ and $wv$ retain their original colors or they exchange their colors between them (See Figure: \ref{fig:switch_move_bla}). We turn to discuss next such steps that decrease $\Phi$.

\begin{figure}[h!]
  \centering
  \includegraphics[width=0.5\textwidth]{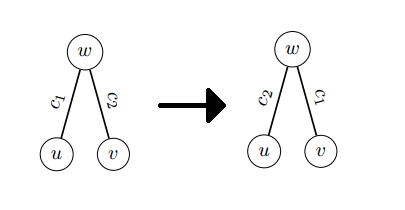}
  \caption{A $(u,v)$-flip as seen from vertex $w$}
  \label{fig:switch_move_bla}  
\end{figure}

\begin{lemma}\label{restricted_recolor}
If $|a_{u,\lambda} - a_{v,\lambda}| \ge 2$ for some color $\lambda$, then there is a $(u,v)$-flip that decreases $\Phi$.
\end{lemma}
\begin{proof}
We use, as above, the notation $x_{\mu}:=a_{u, \mu}$ and $y_{\mu}:=a_{v, \mu}$.
Recall that $\Phi(C)= \sum_{\nu} \Phi(\nu)= \sum_{\nu,\mu} \left(a_{\nu,\mu}(C)\right)^2$. For every $w\neq u, v$, the term $\Phi(w)$ does not change in a $(u,v)$-flip. Let $x'_{\mu}, y'_{\mu}$ be the values of $a_{u, \mu}$ resp. $a_{v, \mu}$ after recoloring. The lemma claims that there is a $(u,v)$-flip for which $$\sum x_{\mu}^2 + y_{\mu}^2 > \sum (x'_{\mu})^2 + (y'_{\mu})^2.$$ This inequality can be equivalently stated as $$\sum (x_{\mu} + y_{\mu})^2 + (x_{\mu} - y_{\mu})^2 >\sum (x'_{\mu} + y'_{\mu})^2 + (x'_{\mu} - y'_{\mu})^2.$$ But it is clear that for every $\mu$ there holds $x_{\mu} + y_{\mu} = x'_{\mu} + y'_{\mu}$ (since a flip preserves the number of edges of any given color that are incident with $u$ or $v$) and so, our claim would follow if we could prove
\begin{claim}\label{simple_improve}
Assuming that  $|x_{\lambda} - y_{\lambda}| \ge 2$ for some color $\lambda$, there is a $(u,v)$-flip such that $|x_{\mu}-y_{\mu}| \ge |x'_{\mu}-y'_{\mu}|$ holds for every color $\mu$ and at least one of these inequalities is strict.
\end{claim}
Claim \ref{simple_improve} is a consequence of the following simple variation on Euler's Theorem:
\begin{claim}\label{euler}
Every oriented multigraph $H$ (possibly with loops) can be reoriented in such a way that for every vertex the indegree and outdegree differ at most by one. Such a reorientation can be found in polynomial time.
\end{claim}
\begin{proof}
If $H$ contains a directed cycle $C$, its edges will not be reversed. For the sake of the proof, we remove $C$'s edges from $H$. By this removal, the difference between the indegree and outdegree of every vertex stays unchanged. We keep eliminating directed cycles this way until we reach an acyclic directed multigraph, which, for convenience we still call $H$. If there is some vertex $s$ for which $d_H^+(s)-d_H^-(s)\ge 2$, let $P$ be a longest directed path that starts at $s$. It necessarily ends at a {\em sink} $y$ with $d_H^+(y)=0$. If we reverse the orientation of all the edges in $P$, then $d_H^+(s)-d_H^-(s)$ remains nonnegative and goes down by $2$, the quantity $|d_H^+(y)-d_H^-(y)|$ does not increase, and for every intermediate vertex $z$ in $P$, there is no change in $d_H^+(z)-d_H^-(z)$. An identical argument works as we consider a longest directed path that ends at a vertex $t$ with $d_H^-(t)-d_H^+(t)\ge 2$ and necessarily starts at a {\em source} $x$ with $d_H^-(x)=0$. The claim follows.
\end{proof}

We introduce next the following directed multigraph $H$. Its vertex set is $\{1,\ldots,n-1\}$ with vertices that correspond to colors. To every vertex $w\neq u,v$ in $K_n$ there corresponds a directed edge in $H$. If in the original edge-coloring of $K_n$ the edges $wu, wv$ are colored $\alpha$ and $\beta$, then the edge in $H$ that corresponds to $w$ goes from $\alpha$ to $\beta$. Consider a reorientation of $H$ as given by Claim \ref{euler}. Accordingly,  if the directed edge $\alpha\to\beta$ gets reversed, then we switch colors between the edges $wu$ and $wv$ in $K_n$, that are presently recolored $\beta$ and $\alpha$ respectively. 

We can complete the proof of Lemma \ref{restricted_recolor}.
Let us interpret the conclusion of Claim \ref{euler} in the language of edge coloring of $K_n$. The conclusion says that $|x'_{\mu}-y'_{\mu}|\le 1$ for all colors $\mu$. However, as mentioned, there also holds $x_{\mu} + y_{\mu} = x'_{\mu} + y'_{\mu}$, so that $|x_{\mu}-y_{\mu}| \ge |x'_{\mu}-y'_{\mu}|$ for all colors $\mu$. Also, by assumption, $|x_{\lambda} - y_{\lambda}| \ge 2$, so that for $\mu=\lambda$ the inequality is strict. The claim follows.
\end{proof}

Flips thus deal successfully with the case of an IV coloring $C$ that has a Vee$_{\alpha}$ as well as a Vee$^{\alpha}$. Just apply Lemma \ref{restricted_recolor} with $u, v$ the centers of these Vees and $\lambda=\alpha$.

There is one last remaining case to consider. Namely, an IV coloring $C$ that is not an OF$_n$, which contains, for no $\lambda$, both a Vee$_{\lambda}$ and a Vee$^{\lambda}$. By the parity argument mentioned above, it must have at least two Vee$_{\alpha}$'s, say a Vee$_{\alpha}^{\beta}= \set{v_1,v_2,v_3}$ centered at $v_2$, and a Vee$_{\alpha}^{\gamma}$. However, in this last remaining case there are no Vee$^{\alpha}$, Vee$_{\beta}$ or Vee$_{\gamma}$. (We do not care whether $\beta = \gamma$ or not). The situation can be resolved by either a single edge recoloring or a single edge recoloring followed by a flip. Let us recolor the edge $v_1 v_2$ from $\alpha$ to $\beta$. This clearly decreases $\Phi(v_2)$ by $2$, and note that $\Phi(v_1)$ increases by at most $2$, since there are no Vee$_{\beta}$'s in $C$. If this single edge change decreases $\Phi(C)$ we are done. If not, $\Phi$ did not change, and this means that there is now a Vee$_\beta^\alpha$ centered at $v_1$. Together with the Vee$_{\alpha}^{\gamma}$ mentioned above, we are now in a position to apply a flip that reduces $\Phi$. This completes the description of the algorithm and the proof of Theorem \ref{thm:main}.

\subsection{Proof of Theorem \ref{thm:strict}.}\label{pf:strict}
We have essentially everything in place now. As mentioned, in $\cal{D}\rm_n$ the vertex set is the set of all the colorings of $E(K_n)$ with colors $\{1,\ldots,n-1\}$. Adjacency between $C$ and $C'$ means that there are two vertices $u, v \in V(K_n)$ such that every edge in $E(K_n)$ on which the colorings $C$ and $C'$ differ is incident with $u$ or with $v$. The statement of Theorem \ref{thm:strict} follows from that of Theorem \ref{thm:main}. The claim about the algorithm's run time is clear since $\Phi$ cannot exceed $O(n^3)$ and it decreases in every step.\qed

\subsection{Proof of Theorem \ref{thm:weak}.}\label{pf:weak}
The basic idea is to redo the proof of Theorem \ref{thm:strict} piecemeal. The initial phase is the same - At each step we recolor an edge so as to reduce the potential $\Phi$, until a locally optimal coloring is reached. The main difficulty is how to carry out a $(u,v)$-flip operation. A flip is a global change and our purpose is to translate it into a series of local steps of single edge recoloring. Let us return to the proof of Claim \ref{euler} and how it describes a flip in terms of the directed multigraph $H$. That proof shows how to carry out a beneficial flip as a series of steps each of which is the reversal of a directed path $P$ in $H$. This path either starts at a vertex with $d_H^+-d_H^-\ge 2$ and ends at a sink or starts at a source and ends at a vertex with $d_H^--d_H^+\ge 2$. The reversal of such a path decreases $\Phi$, so all that remains is understand what happens if we reverse the edges of $P$ one by one from start to end, where the reversal of an edge amounts to two consecutive steps of edge recoloring (or, a single switch move).
 
We need to translate between relevant parameters of the directed multigraph $H$ and the original coloring. Recall that $a_{u,\mu}$ is the number of $\mu$-colored edges incident with the vertex $u$. Clearly, then, $d^+_H(\mu)=a_{u,\mu}$ for every color $\mu$ other than the color of the edge $uv$. Likewise, $d^-_H(\mu)=a_{v,\mu}$. It follows that $\Phi(u)=\sum_{\mu}(d^+_H(\mu))^2+1$ and $\Phi(v)=\sum_{\mu}(d^-_H(\mu))^2+1$. So suppose that we are presently dealing with the reversal of the directed path $P$ that starts from a vertex (=color) $\alpha$ with $d_H^+(\alpha)=2$, $d_H^-(\alpha)=0$ and ends at a sink $\beta$. We wish to understand how $\Phi$ varies throughout the reversal process. As mentioned above, all the numbers $\Phi(w)$ for $w \neq u, v$ remain unchanged, so we need to monitor only the changes in $\Phi(u)+\Phi(v)$, namely in $\sum_{\mu}(d^+_H(\mu))^2+(d^-_H(\mu))^2$. We consider the change in $\Phi$ relative to its initial value throughout the edge-by-edge reversal process. Concretely, let $\gamma\neq\alpha,\beta$ be some vertex in $P$ and let us calculate the change in $\Phi$ at the moment when we have reversed the section of $P$ from $\alpha$ to $\gamma$, but have not done anything yet with the $\gamma$ to $\beta$ section of $P$. All the resulting change in $\Phi$ is due to the change in $(d^+_H(\alpha))^2+(d^-_H(\alpha))^2+(d^+_H(\gamma))^2+(d^-_H(\gamma))^2$. The statement of the theorem follows since the indegrees and outdegrees in $H$ are bounded by $2$ (since $C$ is locally optimal).
\qed

\section{The Markov Chain Perspective}\label{quo_vadis}

Hill-climbing algorithms allow us to efficiently generate large OFs. However, what we ultimately wish for is an efficient method to {\em uniformly} sample OFs. A sampling method that is efficient, uniform, and transparent enough, can help us understand the typical structure of large OFs, which is what we are trying to accomplish. This elusive sampling mechanism would presumably be a Markov Chains equipped with a Metropolis filter. Such a chain can be viewed as a quantitative version of a hill-climbing method or rather of the strict and mild random walks discussed above. We refer the reader to the beautiful exposition in \cite{levin2009markov}, Chapter 3. To illustrate the methodology we examine the sampling method that it yields through a walk on $\cal{G}\rm_{n}$. To define this walk we need to set a parameter $1> \epsilon>0$. If our current state (vertex) of the walk is the coloring $C$, we pick uniformly at random a neighbor of $C$, say the vertex that corresponds to coloring $C'$. If $\Psi(C')\le \Psi(C)$, we step from $C$ to $C'$. However, if $\Psi(C') > \Psi(C)$, we switch from $C$ to $C'$ only with probability $\epsilon^{\Phi(C') - \Phi(C)}$, and stay put at $C$ with probability $1-\epsilon^{\Phi(C') - \Phi(C)}$. In the limit distribution of the resulting Markov Chain the probability of coloring $C$ is proportionate to $\epsilon^{-\Phi(C)}$. In particular, the limit distribution is constant on all OF$_n$s. There are two main properties that we'd like this chain to have: (i) Rapid mixing and (ii) That the OFs capture enough of the total limit distribution. Whether it is possible to achieve these two goals simultaneously is presently unknown, and we raise the following problem:

\begin{problem}\label{problem:metr}
Is there a Metropolis filter for $\cal{G}\rm_{n}$, under which the resulting Markov Chain is rapidly mixing and the stationary probability of the one-factorizations is at least $n^{-O(1)}$?
\end{problem}

To get better sense of this problem, consider what happens for the extremal values of $\epsilon$. If we set $\epsilon=1$, then we are taking a walk on a Hamming cube $[q]^d$ for $q=n-1$ and $d=\binom n2$. As is well-known, this walk mixes in time $\text{polynomial}(n)$, and the limit distribution is uniform. However, the OF$_n$ occupy only a tiny fraction of  $\exp(-\Omega(n^2))$ of the space. 

At the other extreme, if we let $\epsilon = 0$, then the walk becomes the mild walk on $\cal{G}\rm_{n}$. Now the OFs are sinks so there is no mixing to speak of. Also, for $\epsilon=0$ there is no reason to expect uniform distribution and our simulations for $n=8$ (figure \ref{fig:distribution} and table \ref{way_of_eight}) indicate that the resulting distribution is indeed not uniform.

A closely related problem concerns the random process which starts with a mild walk on $\cal{G}\rm_{n}$ until some one-factorization is reached, say $\cal F$. We then modify the color of each edge randomly, independently and with probability $p$, and restart the mild walk. If $p$ is small, the next OF$_n$ to be reached is most likely $\cal F$ again. On the other hand, setting $p=1$ means a complete restart of the process. We ask:

\begin{problem}
How small can $p=p(n)$ be so that with probability bounded away from zero we reach next a OF other than $\cal F$?
\end{problem}

\subsection{A numerical illustration, n=8}

We ran the algorithm from Theorem \ref{thm:strict} and the mild walk for $n=8$. Each was repeated $10^6$ times. As shown in \cite{DicksonSaffron}, and \cite{wallis2013one}, Chapter 11, whose notation we adopt, there are exactly $6$ isomorphism classes of OF$_8$'s. These references also determine the symmetry groups of the six classes, which yields the expected number of times to sample each class under uniform distribution. The following table compares these three distributions. The figures for the mild walk and the strict algorithm seem similar, but we do not know whether this indicates a real phenomenon.

\usetikzlibrary{calc}
\begin{figure}[h!]
  \centering
  \par\medskip
  \begin{tikzpicture}
  \node (img)  {\includegraphics[scale=0.8]{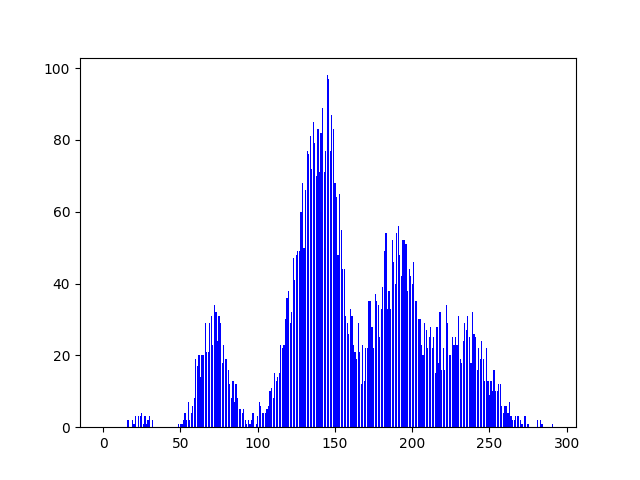}};
  \node[below=of img, node distance=0cm, yshift=1cm,font=\color{black}] {$\#$ Of Times An OF Was Reached};
  \node[left=of img, node distance=0cm, rotate=90, anchor=center,yshift=-0.7cm,font=\color{black}] {$\#$ Of OFs Reached This Amt. Of Times};
 \end{tikzpicture}
	\caption{Limit Distribution of OF$_8$ in the Strict Algorithm.}
  \label{fig:distribution}
\end{figure}

\begin{center}\label{way_of_eight}
\begin{tabular}{ c | c | c | c }
Isomorphism type & Expectation under uniform dist. & Strict Algorithm (Theorem \ref{thm:strict}) & Mild Walk \\ 
\hline
 A & 4807.69 & 547 & 747 \\  
 B & 403846.15 & 355545 & 356265 \\ 
 C & 269230.77 & 305384 & 321701 \\ 
 D & 67307.69 & 66218 & 50959 \\ 
 E & 100961.53 & 40735 & 45058\\
 F & 153846.15 & 231571 & 225270
\end{tabular}
\captionof{table}{Distribution over 1M OF$_8$ Sampled From Our Algorithm}
\end{center}

\section{What we want to know about large OFs}\label{section:expect}

\subsection{Enumeration}

The enumeration of OFs is a challenging problem that has received considerable attention. Of course, we do not expect to have a closed form formula for $f_n$, the number of distinct OF$_n$'s, and the largest $n$ for which $f_n$ is known is $n=14$ \cite{MR2489440}. At present, even the asymptotic behavior of $f_n$ is still not fully resolved, and the current best asymptotic bounds are 
\begin{equation}\label{formula:number}
\left((1+o(1))\cdot\frac{n}{e^2}\right)^{(n^2/2)} \ge f_n \ge \left((1+o(1))\cdot\frac{n}{4e^2}\right)^{(n^2/2)}
\end{equation}
see, e.g., \cite{linial2013upper}. Experts in the field seem confident that the upper bound gives the correct asymptotic value, but this remains unproven at present.

\subsection{In the context of high-dimensional combinatorics}

Let us recall some basic notions in high-dimensional combinatorics. Let $k\ge 2$ be an integer. A {\em $(k-1)$-dimensional permutation} \cite{linial2014upper} is a map $A:[t]^k\to \{0,1\}$ with the following property: For every index $k\ge i \ge 1$ and for every choice of values $t \ge y_j \ge 1$, for all $j\neq i$, there is exactly one value $t\ge z\ge 1$ for which $A(y_1,\ldots,y_{i-1},z, y_{i+1},\ldots,y_{d+1})=1$. We note that a one-dimensional permutation is simply a permutation matrix and that a two-dimensional permutation is synonymous with a Latin square. Also, a OF can be viewed as a Latin square that is symmetric w.r.t.\ the main diagonal, all whose diagonal entries equal $n$. Some of the questions that we consider here extend naturally to the broader context of high-dimensional permutations, e.g., \cite{zbMATH06637028}. Likewise, for Steiner Triple Systems which can be viewed as a Latin square with an even richer symmetry and the same problems and issues apply there as well.

Jacobson and Matthews \cite{jacobson1996generating} have found a Markov Chain whose state set is comprised of all order-$n$ Latin squares. The underlying graph is regular, and as they proved, it is connected, so that the limit distribution is uniform. Whether or not this chain mixes rapidly remains open. 

Consider the strict walk on the graph $\cal{L}\rm_n$ whose $(n!)^{n-1}$ vertices correspond to all $n\times n$ arrays $A$ every row of which is a permutation from $S_n$. The potential $\Psi(A)$ is defined as the number pairs of equal entries in $A$ that reside in the same column. Arrays $A$ and $A'$ are adjacent iff the entries in which they differ belong to only two columns. This yields an efficient hill-climbing generator of order-$n$ Latin squares\footnote{This construction was found in a discussion with Zur Luria. We are grateful to him for his permission to include it here.}, but we do not know how to solve:

\begin{problem}
Find efficient hill-climbing methods to generate permutations in dimensions $3$ and above.
\end{problem}

\subsection{High girth}

The search for high-girth graphs is a long and ongoing saga in modern graph theory. It is closely related to key problems such as the study of expander graphs, to sparsity and small discrepancy in graphs. Similar phenomena are of interest in OF's as well. The most outstanding open question regarding high girth in OF's is Kotzig's well-known conjecture from 1964 \cite{kotzig1963hamilton}.

\begin{conjecture}[Kotzig's perfect OF conjecture]\label{perfect:of}
For every even $n$, there is a one-factorization of $K_n$ in which every two color classes form a Hamilton cycle.
\end{conjecture}

This is known to be true for $n=2p$ and for $n=p+1$ where $p$ is an odd prime, and in an additional finite list of $n$'s.

Erd\H{o}s has defined the {\em girth} of a Steiner Triple System $X$ as the smallest integer $g\ge 5$ such that there is a set of $g$ vertices that contains at least $g-2$ triples of $X$. (Note that the girth of a {\em graph} is the least integer $g\ge 3$ such that there is a set of $g$ vertices than spans $\ge g$ edges). He made the following conjecture:

\begin{conjecture}[Erd\H{o}s STS girth conjecture]
There exist Steiner Triple Systems of arbitrarily high girth.
\end{conjecture}

Our state of knowledge concerning this problem is rather dismal. In particular, not a single triple system is known with girth $\ge 8$. In this view, we raise the following relaxation of the Erd\H{o}s conjecture.

\begin{problem}
Is there a constant $c > 0$ and a family of $3$-uniform $n$-vertex hypergraphs with at least $cn^2$ hyperedges and arbitrarily large girth?
\end{problem}

Coming back to conjecture \ref{perfect:of}, we note that not much seems to be known about the union of three or more  color classes. Moore's bound says that the girth of an $n$-vertex $d$-regular graphs is at most $(2-o_n(1))\cdot\frac{\log n}{\log(d-1)}$. We believe that this  bound is not sharp, and that the coefficient $(2-o_n(1))$ in this bound should be replaceable by a number strictly smaller than $2$. Deciding this problem has turned out to be rather difficult. We raise the following problem in the hope that it offers a viable approach to this question, by way of studying OF's.

Let $d\ge 3$ be an integer, and let $X$ be an OF$_n$. Consider all the $\binom{n-1}{d}$ graphs that can be obtained as the union of some $d$ color classes classes in $X$. If each of these graphs can be made to have girth $(2-o_n(1))\cdot\frac{\log n}{\log(d-1)}$ (i.e., meet Moore's bound) by removing only $O_d(1)$ of its edges, we say that $X$ is {\em $d$-perfect}.

\begin{problem}
Do there exist for every $d\ge 3$, arbitrarily large $d$-perfect one-factorizations?
\end{problem}

For record, we believe that the answer is negative for {\em every} $d\ge 3$, and hope that this may offer a strategy toward showing that the Moore bound is not sharp.

\subsection{Sparsity and low discrepancy}

Concerning expansion and sparsity, many questions suggest themselves.

\subsubsection{The spectrum}

\begin{problem}\label{pr:five}
Is it true that for every $d\ge 3$ and for every even $n > n_0(d)$ there exists a one factorization of $K_n$ such that the union of any set of $d$ color classes forms a Ramanujan graph?
\end{problem}

We note that this question goes substantially beyond what is currently known. It has been a longstanding open question whether arbitrarily large $d$-regular Ramanujan graphs exist for every $d\ge 3$. Substantial numerical evidence suggests that the answer is positive. A possible approach to proving this is provided by the Bilu-Linial {\em signing conjecture} \cite{bilu2006lifts}. In a recent breakthrough \cite{zbMATH06456011} Marcus, Spielman and Srivastava used their theory of interlacing polynomials to establish this conjecture for bipartite graphs. Even more far reaching is the question how likely it is for a large random $d$-regular graph to be Ramanujan. In this context one should mention Friedman's resolution \cite{friedman2008proof} of Alon's conjecture and the more recent work of Bordenave \cite{bordenave2015new} pertaining to these problems. Some simulations related to these questions appear in Figure \ref{fig:test1}.

A less ambitious, but still interesting version of this problem is:

\begin{problem}\label{pr:six}
Is it true that for every $d\ge 3$ and $\epsilon >0$ and for every even $n > n_0(d,\epsilon)$, there exists a one factorization of $K_n$ such that the union of any $d$ color classes forms a graph all of whose nontrivial eigenvalues are between $2\sqrt{d-1}+\epsilon$ and $-\epsilon - 2\sqrt{d-1}$?
\end{problem}

Of course, even much weaker statements would be of great interest, e.g., the same statement with the nontrivial eigenvalues residing in $[-d+\epsilon, d-\epsilon]$.

It is also natural to wonder whether the phenomena considered in problems \ref{pr:five} and \ref{pr:six} hold for {\em asymptotically almost all} one-factorizations.

\begin{figure}[h!]
  \centering
  \includegraphics[width=0.7\textwidth]{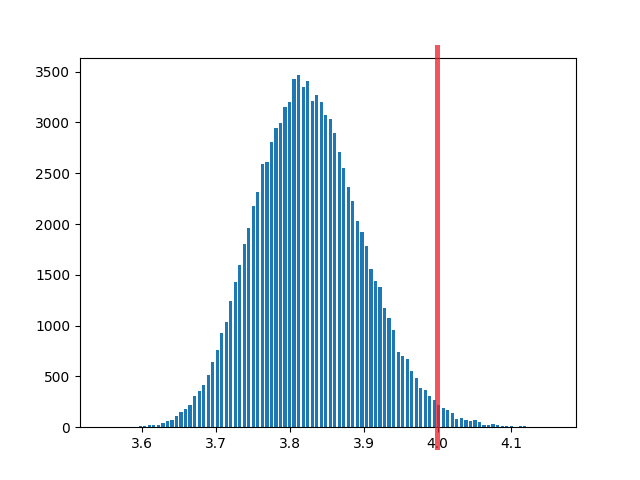}
   \caption{Distribution of the second largest eigenvalue in $10^5$ 5-Regular graphs sampled from OF$_{100}$, produced by the strict algorithm (Theorem \ref{thm:strict}). The vertical red line marks the Ramanujan bound $2\sqrt{5-1}=4$.}
  \label{fig:test1}
\end{figure}

\subsubsection{Sparsity}

For a set of vertices $S$ in an edge-colored $K_n$, let $\gamma(S)$ be the number of different colors of edges in the subgraph induced by $S$.
The {\em deficit} of a $S$ is $D(S):= \binom{|S|}{2} - \gamma(S)$.

We wish to know whether there exist OF's where all small sets have a small deficit. Concretely we ask:

\begin{problem}\label{pr:deficit}
For which $n, k$ and $d$ does it hold that every one-factorization of $K_n$ has a set of $k$ vertices with deficit $\ge d$ ? (We allow dependencies among the parameters $d, k$ and $n$).
\end{problem}
We start with the following easy observation.
\begin{proposition}\label{prop:dfct}
If $S$ is a set of vertices and $l$ of the edges in $E(S)$ are colored by only $c$ colors, then $D(S)\ge l-c$.
\end{proposition}

Here are some results results concerning Problem \ref{pr:deficit}:

\begin{theorem}\label{prop:small:deficit}
For every even integer $n$ and for every $n\ge k\ge 4$, every order-$n$ one-factorization has a set of $k$ vertices with deficit $\ge k-3$.\\
For every integer $d\ge -2$ there is a positive $C$ such that every order-$n$ one-factorization has a set of $k\le C\cdot \log n$ vertices with deficit $\ge k+d$.
\end{theorem}
\begin{proof}
To prove the first claim we use Proposition \ref{prop:dfct} with $c=2$. Namely, we pick any two color classes and consider the subgraph $H\subset K_n$ that they form. Clearly, $H$ is the disjoint union of alternating even cycles, and it therefore has a set of $k$ vertices that spans at least $l=k-1$ edges. The claim follows.

For the second claim of the theorem we again use Proposition \ref{prop:dfct}, now with $c=3$. We also need the following easy proposition whose proof essentially follows from the same argument that establishes the Moore bound.

\begin{proposition}\label{lem:euler_char}
For every positive integer $a$ there is some $b>0$ such that every cubic graph of order $n$ has a set $W$ of at most $b\cdot\log n$ vertices which spans at least $|W|+a$ edges. The bound is tight, up to the dependence of $b$ on $a$.
\end{proposition}

This second part of theorem follows by picking any three colors and applying Proposition \ref{lem:euler_char} to the graph with edges in those colors. The set $W\subseteq V$ has the claimed deficit, where the relevant parameters for Proposition \ref{prop:dfct} are $c=3, l=k+a$ and $d=a-3$.
\end{proof}

\subsection{Asymmetry}

A symmetry of a OF $X$ is specified by a permutation of the vertices $\pi\in S_n$ and a permutation of the colors $\tau\in S_{n-1}$, so that if the edge $i,j$ is colored $k$ in $X$, then the edge $\pi(i),\pi(j)$ is colored $\tau(k)$. We say that $X$ is {\em symmetric} if such $\pi\neq\text{id}$ and $\tau$ exist. Otherwise $X$ is said to be {\em asymmetric}. It is a recurring theme in probabilistic combinatorics that symmetry is rare. For example, there are several theorems which state that in various models of random graphs, asymptotically almost all graphs are asymmetric. The analogous statement for OFs was established by Cameron \cite{cameron2015asymmetric}, but the numerical evidence from \cite{dinitz1994there} suggests that symmetry among OF's is much rarer than what Cameron's bound yields. We raise the following questions:

\begin{problem}
\begin{itemize}
\item
Improve the estimate for the number of symmetric OF$_n$.
\item
Describe the most common symmetries that occur in OFs.
\end{itemize}
\end{problem}

As theoretical computer scientists know very well, it is often hard {\em to find hay in a haystack}, so we ask:

\begin{problem}
Find explicit constructions of large asymmetric one-factorizations.
\end{problem}

\section{More on Generative Hill Climbing Algorithms}\label{achra}

We start with two general comments: Hill climbing can naturally be viewed as a discrete version of the gradient descent method. There is an immense body of knowledge pertaining to gradient methods. It would be very desirable to translate and adapt some of it to the discrete world.\\
There are trivial hill climbers which must clearly be disqualified. For example - Color the edges randomly. If this yields a OF (the chance for which is clearly tiny) - take it. If not, pick a fixed-in-advance OF. One way to rule out such uninteresting methods (see our abstract), is to place an upper bound on the degrees in the underlying graph. In Theorem \ref{thm:weak} the degrees are only $O(n^3)$. For Theorem \ref{thm:strict} the degrees are harder to calculate, but they clearly do not exceed $O((n+2)!)$ which is negligible fraction of all OF$_n$ (see Equation (\ref{formula:number})).\\
It is also possible to rule out such trivial solutions by requiring symmetry under the action of $S_n$.

\subsection{Other heuristics}

We turn next to three heuristics, two from \cite{dinitz1987hill} and one home-grown. All three were tested and seem to have good convergence properties, though nothing rigorous is presently known about them. 

\subsubsection{The Dinitz and Stinson heuristics}
Both hill-climbing heuristics from \cite{dinitz1987hill} are walks on a graph $\cal{DS}\rm_n$, whose vertex set coincides with the set of all properly $(n-1)$-edge colored subgraphs of $K_n$. Concretely, if $H$ is such a subgraph of $K_n$, then every edge of $K_n$ is assigned either a color in $\{1,\ldots,n-1\}$ or is {\em uncolored}. For every $n-1\ge i \ge 1$ the edges colored $i$ form a matching. The potential $\Psi(H)$ is the number of uncolored edges in $H$. Here are their two heuristics.

\begin{heuristic}
If the color $n-1\ge i\ge 1$ is missing at a vertex $x\in V(K_n)$, there must be an uncolored edge incident with $x$, say $xy$. If the color $i$ is missing at $y$ as well, color $xy$ by $i$. Otherwise, say $yz$ is colored $i$, then uncolor $yz$ and color $xy$ by $i$.
\end{heuristic}

\begin{heuristic}
If color $i$ is not a perfect matching, say it is missing at vertices $x, y\in V(K_n)$. Color the edge $xy$ by $i$. This may mean coloring an uncolored edge or recoloring a colored one.
\end{heuristic}

\subsubsection{The "Four-Switch"}\label{abu:arba}

In this heuristic the underlying graph $\cal{M}\rm_n$ has $(n!!)^{n-1}$ vertices. Each vertex $M$ of of $K_n$ represents an ordered list of $(n-1)$ perfect matchings ("color classes") in $K_n$, so that edges of $K_n$ may be multiply colored or be uncolored. The potential $\Psi(M)$ is the number of uncolored edges in $M$.

Adjacency is defined by a {\em four-switch} move. Namely, $M$ and $M'$ are adjacent if there are four vertices $x_1, x_2, x_3, x_4\in V(K_n)$ and a color $\alpha$ such that: In $M$ the edges $x_1 x_2$ and $x_3 x_4$ have the color $\alpha$ whereas in $M'$ the edges $x_2 x_3$ and $x_1 x_4$ have the color $\alpha$. On all other edges of $K_n$ there is perfect agreement between $M$ and $M'$.

To see how the value of $\Psi$ may decrease, note that a vertex of a multiply-colored edge is incident with uncolored edges. A typical case where a four-switch reduces the potential function is shown in Figures \ref{fig:forwards_alg_first_case_before_fix} and \ref{fig:forwards_alg_first_case_after_fix}.

 		\begin{figure}[!h]
 			\begin{minipage}{.5\textwidth}
 				\centering
 				\input{forwards_alg_first_case_before_fix.tex}
 				\caption{Before the switch}
 				\label{fig:forwards_alg_first_case_before_fix}
 			\end{minipage}	
 			\begin{minipage}{.5\textwidth}
 				\centering
 				\input{forwards_alg_first_case_after_fix.tex}
 				\caption{After the four-switch}
 				\label{fig:forwards_alg_first_case_after_fix}
 			\end{minipage}	
 		\end{figure}

\bibliographystyle{plain}
\bibliography{maya}

\end{document}

%% file: First_Fix_Case.tex
\begin{tikzpicture}[ every arrow/.append style={dash,thick}, node distance =2 cm and 2cm, el/.style = {align=left, sloped}, er/.style = {align=right, sloped} ]
    \node[shape=circle,draw=black] (1) at (0,0) {$v_1$};
    \node[shape=circle,draw=black, label={\small $\text{missing } \beta$}] (2) at (.5,1.8) {$v_2$};
    \node[shape=circle,draw=black] (3) at (1,0) {$v_3$};
    \path (1) edge[line width=0.3mm] node[er,above] {$\alpha$} (2);
    \path (2) edge[line width=0.3mm] node[er,above] {$\alpha$} (3); 
    \node[shape=circle,draw=black] (4) at (3,0) {$u_1$};
    \node[shape=circle,draw=black, label={\small $\text{missing }\alpha$}] (5) at (3.5,1.8) {$u_2$};
    \node[shape=circle,draw=black] (6) at (4,0) {$u_3$};
    \path (4) edge[line width=0.3mm] node[er,above] {$\gamma$}(5);
    \path (5) edge[line width=0.3mm] node[er,above] {$\gamma$}(6); 
\end{tikzpicture}

%% file: Second_Fix_Case.tex
\begin{tikzpicture}[ every arrow/.append style={dash,thick}, node distance =2 cm and 2cm, el/.style = {align=left, sloped}, er/.style = {align=right, sloped} ]
    \node[shape=circle,draw=black] (1) at (0,0) {$v_1$};
    \node[shape=circle,draw=black, label={\small $\text{missing } \beta$}] (2) at (.5,1.8) {$v_2$};
    \node[shape=circle,draw=black] (3) at (1,0) {$v_3$};
    \path (1) edge[line width=0.3mm] node[er,above] {$\alpha$} (2);
    \path (2) edge[line width=0.3mm] node[er,above] {$\alpha$} (3); 
    \node[shape=circle,draw=black] (4) at (3,0) {$u_1$};
    \node[shape=circle,draw=black, label={\small $\text{missing }\gamma$}] (5) at (3.5,1.8) {$u_2$};
    \node[shape=circle,draw=black] (6) at (4,0) {$u_3$};
    \path (4) edge[line width=0.3mm] node[er,above] {$\alpha$}(5);
    \path (5) edge[line width=0.3mm] node[er,above] {$\alpha$}(6); 
\end{tikzpicture}

%% file: forwards_alg_first_case_before_fix.tex
\begin{tikzpicture}[auto ,node distance =2 cm and 2cm ,on grid ,	state/.style ={ circle ,draw , text=black , minimum width =0.1 cm},  every loop/.style={}, every arrow/.append style={dash,thick}]
				\node[state] (C) {$a_1$};
				\node[state] (D) [right=of C] {$a_2$};
				\node[state] (A) [above =of C] {$b_1$};
				\node[state] (B) [above =of D] {$b_2$};
				\path (A) edge [red] (B);
				\path (A) edge [dotted] node[left] {Missing-edge} (C);
				\path (B) edge [dotted] node[right] {Missing-edge} (D);
				\path (C) edge [bend left =7, red] (D);
                 \path (C) edge [bend right =7, green] (D);

\end{tikzpicture}

%% file: forwards_alg_first_case_after_fix.tex
\begin{tikzpicture}[auto ,node distance =2 cm and 2cm ,on grid ,	state/.style ={ circle ,draw , text=black , minimum width =0.1 cm},  every loop/.style={}, every arrow/.append style={dash,thick}]
				\node[state] (C) {$a_1$};
				\node[state] (D) [right=of C] {$a_2$};
				\node[state] (A) [above =of C] {$b_1$};
				\node[state] (B) [above =of D] {$b_2$};
				\path (A) edge[dotted] (B);
				\path (C) edge[green] (D);
				\path (A) edge[red] (C);
				\path (B) edge[red] (D);
				
\end{tikzpicture}

%% file: main.bbl
\begin{thebibliography}{10}

\bibitem{bilu2006lifts}
Yonatan Bilu and Nathan Linial.
\newblock Lifts, discrepancy and nearly optimal spectral gap.
\newblock {\em Combinatorica}, 26(5):495--519, 2006.

\bibitem{bordenave2015new}
Charles Bordenave.
\newblock A new proof of {F}riedman's second eigenvalue theorem and its
  extension to random lifts.
\newblock {\em arXiv preprint arXiv:1502.04482}, 2015.

\bibitem{cameron2015asymmetric}
Peter~J Cameron.
\newblock Asymmetric {L}atin squares, {S}teiner triple systems, and
  edge-parallelisms.
\newblock {\em arXiv preprint arXiv:1507.02190}, 2015.

\bibitem{DicksonSaffron}
L.~E. Dickson and F.~H. Safford.
\newblock 8.
\newblock {\em The American Mathematical Monthly}, 13(6/7):150--151, 1906.

\bibitem{dinitz1994there}
Jeffrey~H Dinitz, David~K Garnick, and Brendan~D McKay.
\newblock There are 526,915,620 nonisomorphic one-factorizations of {$K_{12}$}.
\newblock {\em Journal of Combinatorial Designs}, 2(4):273--285, 1994.

\bibitem{dinitz1987hill}
JH~Dinitz and DR~Stinson.
\newblock A hill-climbing algorithm for the construction of one-factorizations
  and room squares.
\newblock {\em SIAM Journal on Algebraic Discrete Methods}, 8(3):430--438,
  1987.

\bibitem{friedman2008proof}
Joel Friedman.
\newblock {\em A proof of Alon's second eigenvalue conjecture and related
  problems}.
\newblock American Mathematical Soc., 2008.

\bibitem{Hall67}
Marshall Hall.
\newblock {\em Combinatorial theory}, volume~71.
\newblock John Wiley \& Sons, 1967.

\bibitem{jacobson1996generating}
Mark~T Jacobson and Peter Matthews.
\newblock Generating uniformly distributed random latin squares.
\newblock {\em Journal of Combinatorial Designs}, 4(6):405--437, 1996.

\bibitem{MR2489440}
Petteri Kaski and Patric R.~J. \"Osterg\aa~rd.
\newblock There are {$1,132,835,421,602,062,347$} nonisomorphic
  one-factorizations of {$K_{14}$}.
\newblock {\em J. Combin. Des.}, 17(2):147--159, 2009.

\bibitem{kotzig1963hamilton}
Anton Kotzig.
\newblock Hamilton graphs and hamilton circuits.
\newblock In {\em Theory of Graphs and its Applications, Proceedings of the
  Symposium of Smolenice}, pages 63--82, 1963.

\bibitem{levin2009markov}
David~Asher Levin, Yuval Peres, and Elizabeth~Lee Wilmer.
\newblock {\em Markov chains and mixing times}.
\newblock American Mathematical Soc., 2009.

\bibitem{linial2013upper}
Nathan Linial and Zur Luria.
\newblock An upper bound on the number of steiner triple systems.
\newblock {\em Random Structures \& Algorithms}, 43(4):399--406, 2013.

\bibitem{linial2014upper}
Nathan Linial and Zur Luria.
\newblock An upper bound on the number of high-dimensional permutations.
\newblock {\em Combinatorica}, 34(4):471--486, 2014.

\bibitem{zbMATH06637028}
Nathan {Linial} and Zur {Luria}.
\newblock {Discrepancy of high-dimensional permutations.}
\newblock {\em {Discrete Anal.}}, 1:8, 2016.

\bibitem{zbMATH06456011}
Adam~W. {Marcus}, Daniel~A. {Spielman}, and Nikhil {Srivastava}.
\newblock {Interlacing families. I: Bipartite Ramanujan graphs of all degrees.}
\newblock {\em {Ann. Math. (2)}}, 182(1):307--325, 2015.

\bibitem{mendelsohn1985one}
Eric Mendelsohn and Alexander Rosa.
\newblock One-factorizations of the complete graph—a survey.
\newblock {\em Journal of Graph Theory}, 9(1):43--65, 1985.

\bibitem{wallis2013one}
Walter~D Wallis.
\newblock {\em One-factorizations}, volume 390.
\newblock Springer Science \& Business Media, 2013.

\end{thebibliography}
